\documentclass[11pt,a4paper,reqno]{amsart}
\usepackage[T1]{fontenc}
\usepackage{mathrsfs}
\usepackage{amsfonts,amssymb,amsmath,amsgen,amsthm}
\usepackage{color}
\usepackage{amsbsy}
\usepackage{mathrsfs}
\usepackage{bbm}
\usepackage{titletoc}
\usepackage{enumerate}
\usepackage{hyperref}
\usepackage{subfigure}
\usepackage[subfigure]{ccaption}
\theoremstyle{plain}
\newtheorem{theorem}{Theorem}[section]
\newtheorem{definition}[theorem]{Definition}

\newtheorem{lemma}[theorem]{Lemma}

\newtheorem{proposition}[theorem]{Proposition}

\newtheorem{remark}[theorem]{Remark}

\catcode`@=11
\def\section{\@startsection{section}{1}%
  \z@{1.5\linespacing\@plus\linespacing}{.5\linespacing}%
  {\normalfont\bfseries\large\centering}}
\catcode`@=12
\def\CC{{\mathbb C}}
\def\RR{{\mathbb R}}
\def\NN{{\mathbb N}}
\def\ZZ{{\mathbb Z}}

\def\({\left(}
\def\){\right)}
\def\<{\left\langle}
\def\>{\right\rangle}

\numberwithin{equation}{section}

\newcommand{\be}{\begin{equation}}
\newcommand{\ee}{\end{equation}}
\newcommand{\bea}{\begin{eqnarray}}
\newcommand{\eea}{\end{eqnarray}}
\newcommand{\bee}{\begin{eqnarray*}}
\newcommand{\eee}{\end{eqnarray*}}
\newcommand{\SC}{\mathcal S}

\def\ds{\displaystyle}

\def\bs{\bigskip}

\def\fref#1{{\rm (\ref{#1})}}
\def\pref#1{{\rm \ref{#1}}}

\def\pa{\partial}

\def\psij{\Psi_j}
\def\phij{\Phi_j}

\def\beq{\begin{equation}}
\def\eeq{\end{equation}}
\def\beqnar{\begin{eqnarray}}
\def\eeqnar{\begin{eqnarray}}
\def\beqnarnn{\begin{eqnarray*}}
\def\eeqnarnn{\end{eqnarray*}}
\def\beqnn{\begin{equation*}}
\def\eeqnn{\end{equation*}}

\catcode`@=11
\def\supess{\mathop{\operator@font sup\hspace{.5mm}ess}}
\catcode`@=12

\title[Quantum transport of electrons in graphene layers]{Analysis of models for quantum transport of electrons in graphene layers}
\author[R. El Hajj]{Raymond El Hajj}
\email{raymond.el-hajj@insa-rennes.fr}
\address{IRMAR, INSA de Rennes}
\author[F. M\'ehats]{Florian M\'ehats}
\email{florian.mehats@univ-rennes1.fr}
\address{IRMAR, Universit\'e de Rennes 1 and INRIA, IPSO Project}
\date{}

\begin{document}

\begin{abstract}\noindent
We present and analyze two mathematical models for the self consistent quantum transport of electrons 
in a graphene layer. We treat two situations. First, when the particles can move in all the plane $\RR^2$, the model takes the form of a system of massless Dirac equations coupled together by a selfconsistent potential, which is the trace in the plane of the graphene of the 3D Poisson potential associated to surface densities.
In this case, we prove local in time existence and uniqueness of a solution in $H^s(\RR^2)$, for $s > 3/8$ which includes in particular the energy space $H^{1/2}(\RR^2)$. 
The main tools that enable to reach $s\in (3/8,1/2)$ are the dispersive Strichartz estimates that we generalized here for mixed quantum states. Second, we consider a situation where the particles are constrained in a regular bounded domain $\Omega$. In order to take into account Dirichlet boundary conditions which are not compatible with the Dirac Hamiltonian $H_{0}$, we propose a different model built on a modified Hamiltonian displaying the same energy band diagram as $H_{0}$ near the Dirac points. The well-posedness of the system in this case is proved in $H^s_{A}$, the domain of the 
fractional order Dirichlet Laplacian operator, for $1/2\leq s<5/2$.
\end{abstract}

\maketitle

\bs

\section{Introduction and main results}	
\subsection{Setting of the model}
The graphene is a two dimensional crystal of carbon. It appears naturally in the graphite crystals, which consist of many graphene sheets stacked together, and was isolated experimentally for the first time in 2004 by A. Geim's team, see \cite{geim}. Actually, discovery of graphene has opened new ways to study some basic quantum relativistic phenomena which have always been considered as very exotic. Indeed, this new and strictly two-dimensional material displays unusual electronic properties arising from the biconically shaped form of the Fermi surfaces near the Brillouin zone corners \cite{castro}.
In a quite wide range of energy, electrons and holes propagate as relativistic massless Fermions with an effective velocity $v_F\simeq \frac{c}{300}$, with a linear dispersion relation $E = \pm\hbar v_F |k|$, and their behavior reproduces the physics of quantum electrodynamics but at the much smaller energy scale of the solid state physics. 

Graphene is thus a zero band-gap semiconductor with a linear, rather than quadratic, long-wavelength energy dispersion for the electron-hole pairs, which cannot be decoupled as usually in semiconductors.

 In this paper, we are interested in the mathematical analysis of two quantum models for the transport of an electron gas in a graphene layer, including many-body effects via a mean-field description. We will indeed consider two situations: the case where the particles can move under the action of the potential {\em in all the plane} $\RR^2$, and the case where they are constrained {\em in a bounded domain} $\Omega$, roughly modeling an electronic device. In this case, we will have to prescribe additionally some boundary conditions, leading to some specific difficulties.

In both models, the electron ensemble is a mixed quantum state, described by a density matrix
$$\sum_{j\in \NN}\lambda_j \left|\Psi_j\right\rangle\left\langle\Psi_j\right|$$
which corresponds to a statistical mixture of pure states $\Psi_j$. The coefficients $\lambda_j>0$ are the occupation probabilities and are fixed during the evolution of the system (we only consider here closed quantum systems, without interaction with an environment). In this description, the valence and conduction bands are denoted as pseudo-spin components of the particle described in term of a vector wave function 
$$\Psi_j=\left(\begin{array}{c}\psi^+_j \\\psi^-_j\end{array}\right)\in \CC^2.$$ 

\subsection*{The whole space case} Let us first introduce the model in the case $\Omega=\RR^2$. The time evolution of each vector wave function $\Psi_j$ is given by the following two-dimensional massless Dirac equation:
\begin{equation}
i\partial_t\psij=- i\sigma\cdot\nabla_x\psij+V\psij,\qquad \psij(0,x)=\phij(x). \label{dirac}
\end{equation}
Here, $\sigma=(\sigma_1,\sigma_2)$ denotes a vector of the two Pauli matrices given by
 \beq\label{paulimatrices}
 \sigma_1=\left(
                                    \begin{array}{cc}
                                      0 & 1 \\
                                      1 & 0 \\
                                    \end{array}
                                  \right),\qquad \sigma_2=\left(\begin{array}{cc}
                                      0 & -i \\
                                      i & 0 \\
                                    \end{array}
                                  \right),
\eeq                               
and $\sigma\cdot\nabla_x=\sigma_1\partial_{x_1}+\sigma_2\partial_{x_2}$. 
We assume that no external potential is applied to the system and that $V=V(t,x)$ is the self-consistent electrostatic potential. 
Since the particles are confined in the graphene plane, the self-consistent potential satisfies the following "confined Poisson equation", which is the trace on the graphene plane of the 3D Poisson equation:
\begin{equation}
\label{poisson}
V(t,x)=\frac{1}{4\pi}\int_{\RR^2}\frac{n(t,y)}{|x-y|}dy,
\end{equation}
where 
$$n(t,x)=\sum_{j\in \NN}\lambda_j|\psij(t,x)|^2$$
 is the bidimensional surface density of particles. Here $|\psij|^2=|\psi^+_j |^2+|\psi^-_j|^2$ and  $\lambda=(\lambda_j)_{j\in \NN}$ is a fixed $\ell^1$ sequence of positive real numbers.  A derivation of this equation \fref{poisson} from the $3D$ Poisson equation for a bidimensionally confined electron gas is sketched in Appendix \ref{appendix}. We notice that \eqref{poisson} is equivalent to the following fractional Laplacian equation in $\RR^2$:
$$
2(-\Delta)^{1/2} V=n.
$$

\sloppy
This Dirac-Hartree system \fref{dirac}, \fref{poisson} is the graphene counterpart of the Schr\"{o}dinger-Poisson system, which is one of the most used models for studying the quantum transport in semiconductor devices \cite{MRS,CAZ,IZL,cetraro}. For mixed quantum states, the $H^2$ and $H^1$ theories of the Schr\"odinger-Poisson system were developed in Ref. \cite{IZL,ARN,BRMA}. F. Castella \cite{Cas} studied the $L^2$ solutions using a generalization, for mixed states, of Strichartz estimates. 

\bs
One of the main differences here is that the Dirac-Hartree system \fref{dirac}, \fref{poisson} does not have a positive energy like Schr\"odinger-Poisson. Indeed, let $H_0=-i\sigma\cdot\nabla_x$. The spectrum of $H_0$ is continuous and given by 
\begin{equation}\label{spectreH0}
\sigma(H_0)=\RR=\{ \pm |\xi|,\quad \xi\in \RR^2  \}.
\end{equation}
 For any $\xi\in \RR^2$, $\xi\neq 0$, let $\Pi_+$ and $\Pi_-$ the spectral projection operators associated to $+|\xi|$ and $-|\xi|$ respectively,
 \begin{equation}\label{projecteur}
 \Pi_{\pm}=\frac12\left({\rm Id}\pm\frac{H_0}{|\xi|}\right).
 \end{equation}
Multiplying \eqref{dirac} by $\lambda_{j}\partial_t\overline\Psi_j$, integrating with respect to $x$, summing on $j$ and taking the real part, one obtains the following energy conservation 
\beq\label{energyconserv}
\frac{d}{dt} E(t)=0,
\eeq   
with
\begin{eqnarray}\label{energy}\nonumber
E(t)&=&\sum_j\lambda_j\int_{\RR^2}\langle -i\sigma \cdot \nabla_x \Psi_j,\Psi_j\rangle dx +\int_{\RR^2}|(-\Delta)^{1/4}V|^2dx,\\
&=&\sum_j\lambda_j\int_{\RR^2}|\xi| \left(|\Pi_+\Psi_j|^2 -|\Pi_-\Psi_j|^2\right)d\xi+\int_{\RR^2}|(-\Delta)^{1/4}V|^2 dx
\end{eqnarray}
where $\langle\cdot,\cdot\rangle$ denotes the Hermitian scalar product in $\CC^2$. Hence, $E(t)$ is not positive and one expects only local in time existing results. 

\subsection*{Case of a bounded domain} Let us now describe our model in the case where the particles are supposed to be constrained to move in a smooth bounded domain $\Omega\subset \RR^2$, with boundary $\pa\Omega$. The question of boundary conditions for the Dirac equation is a delicate issue that has been discussed in the physical literature  \cite{castro, wurm}: several choices of boundary conditions have been proposed, corresponding to different modeling of the boundary. We emphasize that the simple Dirichlet boundary conditions --\,usually modeling infinite walls in semiconductors\,-- are not compatible with the Dirac Hamiltonian $H_0=-i\sigma\cdot\nabla_x$ which is an operator of order 1. Here, we shall adopt another strategy, following for instance \cite{akola,harrell-yolcu} where it has been proposed to modify in an ad-hoc way the relativistic Hamiltonians in order to take into account Dirichlet boundary conditions. The idea, also discussed in \cite{robinson,fuchs}, is to choose a square root of the massless Klein-Gordon operator different from the Dirac operator $H_0$. 

Throughout this paper, we shall denote by $A=-\Delta_{\rm Dir}$ the opposite of the Dirichlet Laplacian on $L^2(\Omega)$, of domain $H^2(\Omega)\cap H^1_0(\Omega)$. We introduce the Hamiltonian
\begin{equation}\label{Htilde0}
\widetilde H_0=  \left(\begin{array}{cc}
                                      0 & A^{1/2} \\
                                      A^{1/2} & 0 \\
                                    \end{array}
                                  \right)=\sigma_1 A^{1/2},
\end{equation}
where $\sigma_1$ is the first Pauli matrix defined by \eqref{paulimatrices} and $A^{1/2}$ is the square root of
the opposite of the Dirichlet Laplacian, spectrally defined, see Subsection \ref{sectionmain} for a precise definition. We point out that, in the whole space case, $\widetilde H_0$ displays the same energy band diagram as $H_0$, which approximates near the so-called Dirac points the graphene band diagram computed in \cite{wallace}. Moreover, as $H_0$ does, the non diagonal structure of $\widetilde H_0$ induces couplings between electrons and holes.

We are now ready to write our model in the case of a bounded domain $\Omega$, that will be built on this modified Hamiltonian $\widetilde H_0$. For all $j\in \NN$, the 2-spinor wave function $\Psi_j(t,x)$ defined on $\RR\times \Omega$ solves the following equation:
\begin{equation}
i\partial_t\psij=\widetilde H_0\psij+V\psij,\qquad \psij(0,x)=\phij(x), \label{dirac2}
\end{equation}
where the self-consistent potential $V$ is the solution of the confined Poisson equation with Dirichlet boundary conditions, formally derived in Appendix \ref{appendix}:
\beq\label{poisson2}
V=\frac{1}{2}A^{-1/2}n=\frac{1}{2}(-\Delta_{\rm Dir})^{-1/2}n
\eeq  
with
$$n(t,x)=\sum_{j\in \NN}\lambda_j|\psij(t,x)|^2.$$

\bs
The outline of this article is as follows. In the next subsection of this introduction, we present our two main results. First, in the whole-space case, we prove the local in time existence and uniqueness of a solution to \fref{dirac}, \fref{poisson} in $H^s$, for $ s>\frac{3}{8}$. This result is stated in Theorem \ref{main-theorem2} and proved in Section \ref{wholedomain}. The main tools that enable to reach $s\in (\frac{3}{8},\frac{1}{2})$ are the dispersive Strichartz estimates, that we generalize here for mixed quantum states (in Subsection \ref{wholedomain2}), analogously to the work of F. Castella \cite{Cas} for the Schr\"odinger-Poisson system. Second, in the case of a bounded domain, we prove the local in time existence and uniqueness of a solution to \fref{dirac2}, \fref{poisson2} in $H^s$, for $\frac12\leq s<\frac{5}{2}$. In particular, this shows that the problem is well-posed in the energy space $H^{1/2}$. This result is stated in Theorem \ref{main-theorem1} and proved in Section \ref{boundeddomain}. 
Finally, in Appendix \ref{appendix}, we derive formally \fref{poisson} as the trace of the 3D Poisson equation on the graphene plane and, in Appendix \ref{appB}, we recall a few useful properties of fractional order Sobolev spaces in bounded domains.

To end this introduction, we refer the reader to \cite{ESVG,ESSE,MOR,MATSU1,MATSU2} and references therein for the study of semilinear and non linear Dirac type equations. In particular, the existence of ground states for a Hartree-Fock graphene model in the whole space case is proved in \cite{hainzl} and a Dirac-Gross-Pitaevskii model is derived in \cite{fefferman}.

\subsection{Main results}
\label{sectionmain}

\subsection*{The whole space case} Let us present our main result for \eqref{dirac}, \eqref{poisson}.
As in \cite{Cas}, we need to introduce vector-valued spaces adapted to mixed quantum states. 
\begin{definition}\label{SpacesDef2}
For all $s\in \RR$, we denote
$$H^s(\lambda)=\{\Phi=(\Phi_j)_{j\in \NN}\in (H^s(\RR^2,\CC^2))^\NN\,:\,\|\Phi\|_{H^s(\lambda)}^2=\sum_j\lambda_j \|\Phi_j\|_{H^s}^2< +\infty\},$$
where $H^s(\RR^2,\CC^2)$ is the usual Sobolev space of exponent $s$ for spinor wave functions.
\end{definition}
The first main result of this article is the following.
\begin{theorem}\label{main-theorem2}
(i) Let $s\geq\frac{1}{2}$. Then there exists $T >0$ such that, for all $\Phi\in H^s(\lambda)$, the Cauchy problem \fref{dirac}, \fref{poisson}  admits a unique solution $\Psi\in C^0([-T,T],H^s(\lambda))$.\\[1mm]
(ii) Let $ \frac{3}{8}<s<\frac{1}{2}$. Then there exists $T >0$ and a subspace $X_T$ of $C^0([-T,T],H^s(\lambda))$ (see Definition \eqref{defXT}) such that, for all $\Phi\in H^s(\lambda)$, the Cauchy problem \fref{dirac}, \fref{poisson}  admits a unique solution $\Psi\in X_T$.
\end{theorem}
 The proof of this theorem is given in Section \ref{wholedomain}.

\subsection*{Case of a bounded domain $\Omega$} 
Let us now describe the functional framework and our main result for the case of a bounded domain $\Omega\subset \RR^2$. We assume that $\Omega$ is regular, i.e. that its boundary $\pa\Omega$ is an infinitely differentiable manifold in $\RR$ and that $\Omega$ is locally sited in the same side with respect to $\pa\Omega$.

Recall that $A=-\Delta_\Omega$ is the opposite of the Dirichlet Laplacian on $L^2(\Omega,\CC^2)$, of domain $H^2(\Omega)\cap H^1_0(\Omega)$. Let $(e_p)_{p\in\NN}$ denote an orthonormal basis of eigenvectors of $A$, and let $(\mu_p)_{p\in\NN}$, $\mu_p >0$,  be the associated eigenvalues. The most convenient scale of Sobolev spaces adapted to the resolution of our problem is the scale of domains of the fractional order Dirichlet Laplacian operator, that we shall denote by
\be
\label{Hsnorm}
H^s_A=D(A^{s/2})=\{u\in L^2(\Omega,\CC^2)\,:\; \|u\|_{H^s_A}^2=\sum_{p\in \NN}\mu_p^s \,|\langle u,e_p\rangle|^2<+\infty\},
\ee
for $s\in \RR_+$. We recall in Appendix \ref{appB} several properties of these spaces that are useful in this article.

The domain of the self-adjoint operator $\widetilde H_0$ defined by \eqref{Htilde0} is $D(\widetilde H_0)=H^{1}_A$ and for all $\Psi\in D(\widetilde H_0)$, one has
$$\widetilde H_0 \Psi=\sum_{p\in \NN}\mu_p^{1/2} \sigma_1\Psi_p e_p(x),\qquad \mbox{where}\quad \Psi(x)=\sum_{p\in \NN}\Psi_p\,e_p(x).$$
Moreover, for all surface density $n\in L^2(\Omega)$, the selfconsistent potential $V$ written in \fref{poisson2} is precisely defined by
$$V(t,x)=\frac{1}{2}\sum_{p\in \NN}\mu_p^{-1/2}n_p(t)\,e_p(x),\qquad \mbox{where}\quad n(t,x)=\sum_{p\in \NN}n_p(t)\,e_p(x).$$

As above, we introduce the vector-valued functional space adapted to our problem. Recall that $\lambda=(\lambda_j)_{j\in \NN}$ is a fixed $\ell^1$ sequence of positive real numbers.
\begin{definition}\label{SpacesDef}
Let $s>0$, we denote
$$H^s_A(\lambda)=\{\Phi=(\Phi_j)_{j\in \NN}\in (H^s_A)^\NN\,: \; \|\Phi\|_{H^s_A(\lambda)}^2=\sum_j\lambda_j \|\Phi_j\|_{H^s_A}^2< +\infty \}.$$
\end{definition}
We are now able to state our result concerning the case of a bounded domain.
\begin{theorem}\label{main-theorem1}
Let $\frac12\leq s<\frac{5}{2}$. Then there exists $T > 0$ such that, for all $\Phi \in H^s_A(\lambda)$, the Cauchy problem \fref{dirac2}, \fref{poisson2}  admits a unique solution in $\Psi\in C^0([-T,T],H^s_A(\lambda))$.
\end{theorem}

This result is proved in Section \ref{boundeddomain}.
 
    
\section{The case $\Omega=\RR^2$: proof of Theorem \ref{main-theorem2}}\label{wholedomain}   

In this section, we study \eqref{dirac}, \eqref{poisson} and prove Theorem \ref{main-theorem2}. We proceed into two steps. In a first subsection, we prove Item {\em (i)}: the well-posedness of the Cauchy problem in $H^s(\RR^2)$, for all $s\geq \frac 12$. Then, in a second subsection, we prove Item {\em (ii)} and treat the cases $\frac38<s<\frac12$, for which we need to derive suitable Strichartz inequalities.

\subsection{Proof of Theorem \ref{main-theorem2} for $\mathbf {s\geq \frac 12}$ (Item \mbox{\em(i)})}

Let us introduce  the unitary group $K(t)$ defined on $L^2(\RR^2,\CC^2)$ by
\beq\label{kernel} 
K(t)=\cos(t(-\Delta)^{1/2}){\mathrm I}_2-i(\sigma\cdot\nabla_x)(-\Delta)^{-1/2} \sin(t(-\Delta)^{1/2}),
\eeq
where ${\mathrm I}_2$ is the 2x2 identity matrix and $\sigma$ is defined by \eqref{paulimatrices}. It is straightforward to see that the operators $K(t)$ and $(1-\Delta)^{s/2}$ commute together, hence $K(t)$ is also an isometry on any $H^s(\RR^2,\CC^2)$. Moreover, $K(t)$ is the free propagator generated by the linear Dirac operator. Indeed, differentiating the linear Dirac equation
$$\partial_t \psi=-\sigma\cdot\nabla_x\psi,\quad \psi(0)=\phi$$
with respect to the time variable yields the linear wave equation
\beq\label{linearWave}
\partial_t ^2\psi=\Delta\psi
\eeq
with the Cauchy data
\beq\label{linearWaveData}
 (\psi(0),\partial_t\psi(0))=(\phi,-\sigma\cdot\nabla\phi).
 \eeq
 The solution of \eqref{linearWave}, \eqref{linearWaveData} can represented as
 $$\psi(t)=K(t)\phi.$$

Now, our problem \fref{dirac}, \fref{poisson}  can be rewritten under the following mild formulation
\beq \label{integral-form}
\Psi_j(t)=K(t) \Phi_j + \int_0^tK(t-t')F_j(\Psi(t'))dt'.\eeq
Here, we have denoted $F_j(\Psi)=V(\Psi)\Psi_j$, where the selfconsistent potential $V(\Psi)$ associated to a vector wavefunction $\Psi=(\Psi_j)_{j\in \NN}$ is defined by
\begin{equation}
\label{poissonbis}
V(\Psi)(x)=\frac{1}{4\pi}\sum_{j\in \NN}\lambda_j\int_{\RR^2}\frac{|\Psi_j(y)|^2}{|x-y|}dy.
\end{equation}

Let $s\geq\frac{1}{2}$ and $\Phi\in H^s(\lambda)$. In order to solve \fref{integral-form}, for $T>0$ we introduce the mapping $S=(S_j)_{j\in \NN}$ on $C^0([-T,T],H^{s}(\lambda))$ defined by 
 \begin{equation}\label{Adef}
 S_j(\Psi)(t)=K(t) \Phi_j + \int_0^tK(t-t')F_j(\Psi(t'))dt',
 \end{equation}
Let $M > 2 \|\Phi\|_{H^s(\lambda)}$ and define
 \beq\label{B_Mdef}
 B_M= \{ \Psi \in  C^0([-T,T],H^{s}(\lambda))\,:\; \max_{t\in[-T,T]}\|\Psi(t)\|_{H^{s}(\lambda)}\leq M\}.\eeq
To prove Item $(i)$ of Theorem \ref{main-theorem2} (the case $s\geq \frac 12$), it is sufficient to show that $S$ is a contraction mapping on $B_{M}$ for $T>0$ small enough.  Let $\Psi(t)=(\Psi_j(t))_{j\in \NN} \in B_M$. Since $K(t)$ is a unitary group on $H^s(\RR^2)$, we have clearly
$$\max_{t\in[-T,T]}\|S(\Psi)(t)\|_{H^s(\lambda)} \leq \|\Phi\|_{H^s(\lambda)}+\int_{-T}^T\|F(\Psi)(\tau)\|_{H^s(\lambda)}d\tau,$$
where $F(\Psi)=(F_j(\Psi))_{j\in\NN}=(V(\Psi)\Psi_j)_{j\in\NN}$. Now, we claim that, for all $\Psi\in H^s(\lambda)$, one has
\be
\label{claim}\|F(\Psi)\|_{H^s(\lambda)}\lesssim \|\Psi\|_{H^s(\lambda)}^3.
\ee
Assuming this claim, we deduce that, for $\Psi \in B_M$, there exists a constant $C_1 > 0$ such that
 $$ \max_{t\in[-T,T]}\|S(\Psi)(t)\|_{H^s(\lambda)} \leq \|\Phi\|_{H^s(\lambda)}+C_1T\max_{t\in[-T,T]}\|\Psi(t)\|_{H^s(\lambda)}^3.$$
Similarly, for $\Psi, \widetilde\Psi \in B_M$, one can prove that
$$  \max_{t\in[-T,T]}\|(S(\Psi)-S(\widetilde \Psi))(t)\|_{H^s(\lambda)} \leq C_2M^2T\max_{t\in[-T,T]}\|(\Psi-\widetilde \Psi)(t)\|_{H^s(\lambda)}.$$
 Then, if we choose $T$ such that 
 $$\max(C_1,C_2)M^2T\leq \frac12,$$
 recalling that $M \geq 2 \|\Phi\|_{H^s(\lambda)}$,
 one deduces that $S$ is a contraction mapping on $B_M$ and Theorem \ref{main-theorem1} is proved. 

 It remains to prove \fref{claim}. From paradifferential calculus and the Mikhlin-H\"ormander multiplier theorem \cite{ALG,SCH}, one has, for all $j$,
 \begin{equation}
 \label{MH}
 \|F_j(\Psi)\|_{H^s}=\|V(\Psi)\Psi_j\|_{H^s}\lesssim \|V(\Psi)\|_{W^{s,4}}\|\Psi_j\|_{L^4}+\|V(\Psi)\|_{L^\infty}\|\Psi_j\|_{H^s}.
 \end{equation}
 Let us estimate the first term in the right-hand side of \eqref{MH}. From \eqref{poissonbis}, the Minkowski inequality and Hardy-Littlewood-Sobolev, we have
 \begin{align*}
 \|V(\Psi)\|_{W^{s,4}}=\|(1-\Delta)^{s/2}V(\Psi)\|_{L^4}&\lesssim \sum_{j\in \NN}\lambda_j\left\|\int_{\RR^2}\frac{\left((1-\Delta)^{s/2}|\Psi_j|^2\right)(y)}{|x-y|}dy\right\|_{L^4}\\
 &\lesssim \sum_{j\in \NN}\lambda_j\left\|(1-\Delta)^{s/2}\left(|\Psi_j|^2\right)\right\|_{L^{4/3}}\\
 &\lesssim \sum_{j\in \NN}\lambda_j\left\|\Psi_j\right\|_{H^s}\left\|\Psi_j\right\|_{L^4}.
  \end{align*}
 Note that, for the last line, we used again Mikhlin-H\"ormander. Hence, from the Sobolev embedding $H^{1/2}(\RR^2)\hookrightarrow L^4(\RR^2)$, Definition \ref{SpacesDef2} and $s\geq 1/2$, one gets
 \begin{equation}
 \label{MH1}
 \|V(\Psi)\|_{W^{s,4}}\|\Psi_j\|_{L^4}\leq \|\Psi\|_{H^s(\lambda)}^2\|\Psi_j\|_{H^s}.
 \end{equation}
 For the second term in \eqref{MH}, we write
 \begin{align}
 \|V(\Psi)\|_{L^\infty}&=\frac{1}{4\pi}\supess_{x\in \RR^2}\sum_{j\in \NN}\lambda_j\int_{\RR^2}\frac{|\Psi_j(y)|^2}{|x-y|}dy\lesssim\sum_{j\in \NN}\lambda_j\supess_{x\in \RR^2}\int_{\RR^2}\frac{|\Psi_j(x-y)|^2}{|y|}dy\nonumber\\
&\lesssim\sum_{j\in \NN}\lambda_j\supess_{x\in \RR^2}\int_{\RR^2}|(-\Delta_y)^{1/4}\Psi_j(x-y)|^2dy\nonumber\\
&\quad= \sum_{j\in \NN}\lambda_j\|\Psi_j\|_{\dot H^{1/2}}^2\leq \|\Psi\|_{H^s(\lambda)}^2\label{MH2}
 \end{align}
 where we used the Hardy inequality \cite{tao} in $\dot H^{1/2}(\RR^2)$, the homogeneous Sobolev space \cite{BELO}.
 Finally, inserting \eqref{MH1} and \eqref{MH2} in \eqref{MH} yields
 $$\|F_j(\Psi)\|_{H^s}\lesssim\|\Psi\|_{H^s(\lambda)}^2\|\Psi_j\|_{H^s}$$
 which directly implies \eqref{claim}. The proof of Theorem \ref{main-theorem2} in the case $s\geq 1/2$ (Item $(i)$) is complete.
 \qed

\subsection{Proof of Theorem \ref{main-theorem2} for $\mathbf {\frac38<s<\frac12}$ (Item \mbox{\em(ii)})}\label{wholedomain2}

In order to prove Item {\em (ii)} of Theorem \ref{main-theorem2}, we first derive a generalization of the Strichartz inequalities in the case of weighted Besov spaces, analogously to \cite{Cas} in the case of the Schr\"odinger equation (on this subject, we also refer the reader to the recent work \cite{frank} on Strichartz inequalities for orthonormal functions). 

Let us introduce the adapted functional framework. We first recall the definition of the Besov spaces, see e.g. \cite{BELO}.  Let $\varphi\in \mathcal{S}(\RR^2)$ (the Schwartz space) such that supp$\varphi=\{\xi:\, 2^{-1}\leqslant |\xi|\leqslant 2\}$, $\varphi(\xi) >0$ for $2^{-1}<|\xi|< 2$ and $\sum_{k\in \ZZ}\varphi(2^{-k}\xi)=1, (\xi\neq 0)$. We denote by $\{\varphi_k\}_{0 \leq k<+\infty}$ the Littlewood-Paley dyadic decomposition basis defined by $\mathcal{F}\varphi_k(\xi)=\varphi(2^{-k}\xi), \; 1 \leq k<+\infty$ and $\mathcal{F}\varphi_0(\xi)=1-\sum_1^{+\infty}\varphi(2^{-k}\xi)$ where $\mathcal{F}$ denotes the Fourier transform.
\begin{definition}\label{BesovDef}
Let $s\in \RR$, $1\leq p,r\leq +\infty$. We recall the usual definition
$$\|u\|_{B^s_{p,r}}=\|\varphi_0\ast u\|_p+\left( \sum_{k=1}^{+\infty}2^{skr}\|\varphi_k\ast u\|_p^r  \right)^{1/r}$$
where $\|\cdot\|_p$ denotes the $L^p(\RR^2,\CC^2)$ norm. The Besov space is defined by
$$B^s_{p,r}=\{u\in \SC'(\RR^2)\,:\, \|u\|_{B^s_{p,r}}< +\infty \}.$$
For vectorial functions $\Phi=(\phi_j)_{j\in \NN}$ and for all $\ell^1$ sequence of positive real numbers $\lambda=(\lambda_j)_{j\in \NN}$, we define the corresponding Besov space
$$B^s_{p,r}(\lambda)=\{\Phi=(\Phi_j)_{j\in \NN}\,:\,\|\Phi\|_{B^s_{p,r}(\lambda)}^2=\sum_j\lambda_j \|\Phi_j\|_{B^s_{p,r}}^2< +\infty\}.$$
Finally, for $T> 0, q\in[1,+\infty] $, and for any Banach space $X$, we denote 
$$L^q_TX=L^q([-T,T], X).$$
\end{definition}
Recall that the space $B^s_{2,2}$ can be identified with the usual Sobolev space $H^s$.

\bs
Denote by $K_{\pm}(t)=\exp(\pm it (-\Delta)^{1/2})$ be the two operators composing the Dirac propagator $K(t)$ defined by \eqref{kernel}. These two operators satisfy the following Strichartz estimates measured in the vector-valued spaces introduced above.
\begin{proposition}[\textbf{Strichartz estimates}]\label{Strichartz}
For  $2\leq q_j,r_j\leq  +\infty$ such that
\beq\label{admissible-relation}
\frac{2}{q_j}=\frac{1}{2}-\frac{1}{r_j}, \qquad 2s_j=3\left(\frac{1}{2}-\frac{1}{r_j}\right),
\eeq
and for all $T\in (0,\infty]$, the following estimates hold
\beq\label{strichratz1}
\|K_{\pm}(t)u\|_{L^{q_1}_TB^{-s_1}_{r_1,2}(\lambda)}\lesssim \|u\|_{L^2(\lambda)}
\eeq
and
\beq\label{strichratz2}
\left\|\int_{t'<t}K_{\pm}(t-t')f(t')dt'\right\|_{L^{q_2}_TB^{-s_2}_{r_2,2}(\lambda)}\lesssim \|f\|_{L^{q_3'}_TB^{s_3}_{r_3',2}(\lambda)}
\eeq
where $p'$ denotes the dual exponent to $p$ defined by $1/p+1/p'=1$.
\end{proposition}
The norm $L^2(\lambda)$ in the right-hand of \eqref{strichratz1}, is defined as the $H^s(\lambda)$ and $B^s_{p,q}(\lambda)$ norms by
$$\|\Phi\|_{L^2(\lambda)}^2=\sum_j\lambda_j \|\Phi_j\|_{L^2}^2\,.$$
\begin{proof}
This proof is a combination of Minkowski's inequality and the standard Strichartz estimates for the wave equation. Recall first that, for $s_i$, $q_i$, $r_i$ as satisfying \fref{admissible-relation} and for functions $v$ and $h$ from $\RR^2$ to $\CC^4$, one has \cite{GIVE2,Mach,MATSU2}
\beq\label{strichratz1scal}
\|K_{\pm}(t)v\|_{L^{q_1}_TB^{-s_1}_{r_1,2}}\lesssim \|v\|_{2}
\eeq
and
\beq\label{strichratz2scal}
\left\|\int_{t'<t}K_{\pm}(t-t')h(t')dt'\right\|_{L^{q_2}_TB^{-s_2}_{r_2,2}}\lesssim \|h\|_{L^{q_3'}_TB^{s_3}_{r_3',2}}.
\eeq
Hence, for $u=(u_j)_{j\in \NN}$
\bee
\|K_{\pm}(t)u\|_{L^{q_1}_TB^{-s_1}_{r_1,2}(\lambda)}^2&=&\left(\int_{-T}^{T} \left(\sum_j \lambda_j\|K_{\pm}(t)u_j\|_{B^{-s_1}_{r_1,2}}^2\right)^{q_1/2}dt\right)^{2/q_1}\\
&=&\left\|\sum_j \lambda_j\|K_{\pm}(\cdot)u_j\|_{B^{-s_1}_{r_1,2}}^2\right\|_{q_1/2}\\
&\leq&\sum_j \lambda_j\left\|\|K_{\pm}(\cdot)u_j\|_{B^{-s_1}_{r_1,2}}^2\right\|_{q_1/2}=\sum_j \lambda_j\|K_{\pm}(\cdot)u_j\|_{L^{q_1}_TB^{-s_1}_{r_1,2}}^2\\
&\lesssim&\sum_j \lambda_j\|u_j\|_{2}^2=\|u\|_{L^2(\lambda)}^2
\eee
where we have used the Minkowski's inequality \cite{ineq} (notice that $q_1/2\geq 1$) and the Strichartz estimate \fref{strichratz1scal}. This proves \fref{strichratz1}.
Let us prove \fref{strichratz2}. For $f=(f_j)_{j\in \NN}$, denoting $$g_j(t)=\int_{t'<t}K_{\pm}(t-t')f_j(t')dt',$$ we estimate similarly (we have also $q_2/2>1$)
\bee
&&\left\|\int_{t'<t}K_{\pm}(t-t')f(t')dt'\right\|_{L^{q_2}_TB^{-s_2}_{r_2,2}(\lambda)}^2=\left(\int_{-T}^{T} \left(\sum_j \lambda_j\|g_j(t)\|_{B^{-s_2}_{r_2,2}}^2\right)^{q_2/2}dt\right)^{2/q_2}\\
&&\qquad \leq \sum_j \lambda_j\|g_j\|_{L^{q_2}_TB^{-s_2}_{r_2,2}}^2\lesssim\sum_j \lambda_j \|f_j\|_{L^{q_3'}_TB^{s_3}_{r_3',2}}^2=\sum_j \lambda_j\left\|\|f_j\|_{B^{-s_3}_{r_3',2}}^2\right\|_{q_3'/2}\\
&&\qquad \qquad \qquad \qquad \qquad \qquad \qquad \qquad\leq \left\|\sum_j \lambda_j\|f_j\|_{B^{-s_3}_{r_3',2}}^2\right\|_{q_3'/2}=\|f\|_{L^{q_3'}_TB^{s_3}_{r_3',2}(\lambda)}^2
\eee
where we have used \fref{strichratz2scal} and the reverse Minkowski's inequality \cite{ineq} (note that we have necessarily $q_3>2$, so $q_3'/2< 1$). The proof of the Proposition is complete.
\end{proof}

\begin{proof}[Proof of Item $(ii)$ of Theorem \pref{main-theorem2} (the case $\frac38<s<\frac12$).] Let $ s\in(\frac{3}{8},\frac12)$ and $\Phi\in H^s(\lambda)$. Let us fix once for all $r,q, \sigma$ satisfying
\begin{equation}\label{relation}
r>\frac{1}{4s-3/2},\qquad\frac{2}{q}=\frac{1}{2}-\frac{1}{r},\qquad
2\sigma=3\left(\frac{1}{2}-\frac{1}{r}\right).
\end{equation}
We are now ready to define the subspace $X_T\subset C^0([-T,T],H^s(\lambda))$ used in Theorem \ref{main-theorem2}. For $T>0$, define
\begin{equation}
\label{defXT}X_T=C^0([-T,T],H^s(\lambda))\bigcap L^q_TB^{s-\sigma}_{r,2}(\lambda),
\end{equation}
with
$$\|\Psi\|_{X_T}=\|\Psi\|_{L^\infty_TH^s(\lambda)}+\|\Psi\|_{L^q_TB^{s-\sigma}_{r,2}(\lambda)}.$$
We introduce the following mapping $S$ on $X_T$:
\be
\label{defA}
S(\Psi)_j(t)=K(t) \Phi_j + \int_0^tK(t-t')F_j(\Psi(t'))dt',
\ee
where the operator $K(t)$ is defined by \fref{kernel}. Mild solutions to our problem are fixed-points of $S$. For $M> 2C_1\|\Phi\|_{H^s(\lambda)}$, where $C_1$ is a universal constant that is precised below, we denote
$$\mathcal B_M=\{\Psi \in X_T\,:\, \|\Psi\|_{X_T}\leq M\}.$$
Let us prove that $S$ is a contraction mapping on $\mathcal B_M$ for some $T>0$ small enough. Using \eqref{kernel} and Strichartz estimates given in Proposition \ref{Strichartz}, we get
\begin{align} 
\|S(\Psi)\|_{X_T} \lesssim &\|\Phi\|_{H^s(\lambda)}+\|F(\Psi)\|_{L^1_TH^s(\lambda)}\\
&+\|(\sigma\cdot\nabla_x)(-\Delta)^{-1/2} \Phi\|_{H^s(\lambda)}+\|(\sigma\cdot\nabla_x)(-\Delta)^{-1/2} F(\Psi)\|_{L^1_TH^s(\lambda)}\nonumber\\
\lesssim &\|\Phi\|_{H^s(\lambda)}+\|F(\Psi)\|_{L^1_TH^s(\lambda)},
\label{estimApsi}
\end{align}
where we used that the operator $(\sigma\cdot\nabla_x)(-\Delta)^{-1/2} $ is bounded in any $H^s$. 

Now we claim that, under the assumption $\frac{3}{8}<s< \frac{1}{2}$ and if $r,q,\sigma$ have been chosen in order to satisfy \fref{relation}, then for all $u_1,u_2, u_3 \in H^s\cap B^{s-\sigma}_{r,2}$ we have the estimate
\bea
\nonumber
\left\|{\mathcal V}(u_1u_2)u_3\right\|_{H^s}&\lesssim& \|u_1\|_{H^s}\|u_2\|_{H^s}\|u_3\|_{B^{s-\sigma}_{r,2}}+\|u_2\|_{H^s}\|u_3\|_{H^s}\|u_1\|_{B^{s-\sigma}_{r,2}}\\
&&+\|u_3\|_{H^s}\|u_1\|_{H^s}\|u_2\|_{B^{s-\sigma}_{r,2}}
\label{cl2}
\eea
where we have denoted
$${\mathcal V}(u)=\frac{1}{4\pi|x|}*u.$$
Assuming this claim, let us end the proof of the Theorem. First, we deduce from \fref{cl2} that, for all $\Psi_1=(\Psi_{1j})_{j\in \NN}$, $\Psi_2=(\Psi_{2j})_{j\in \NN}$, $\Psi_3=(\Psi_{3j})_{j\in \NN}$ belonging to the space $H^s(\lambda)\cap B^{s-\sigma}_{r,2}(\lambda)$,
\bea
\nonumber
&&\hspace*{-4ex}\left\|{\mathcal V}\left(\sum_j\lambda_j\Psi_{1j}\Psi_{2j}\right)\Psi_3\right\|_{H^s(\lambda)}^2= \sum_k\lambda_k\left\|\sum_j\lambda_j {\mathcal V}(\Psi_{1j}\Psi_{2j})\Psi_{3k}\right\|_{H^s}^2\nonumber\\
&&\leq \sum_k\lambda_k\left(\sum_j\lambda_j\left\|{\mathcal V}(\Psi_{1j}\Psi_{2j})\Psi_{3k}\right\|_{H^s}\right)^2\nonumber\\
&&\lesssim \sum_k\lambda_k\left(\sum_j\lambda_j\|\Psi_{1j}\|_{H^s}\|\Psi_{2j}\|_{H^s}\right)^2\|\Psi_{3k}\|_{B^{s-\sigma}_{r,2}}^2\nonumber\\
&&\quad + \sum_k\lambda_k\left(\sum_j\lambda_j\|\Psi_{1j}\|_{B^{s-\sigma}_{r,2}}\|\Psi_{2j}\|_{H^s}\right)^2\|\Psi_{3k}\|_{H^s}^2\nonumber\\
&&\quad +\sum_k\lambda_k\left(\sum_j\lambda_j\|\Psi_{1j}\|_{H^s}\|\Psi_{2j}\|_{B^{s-\sigma}_{r,2}}\right)^{2}\|\Psi_{3k}\|_{H^s}^2\nonumber\\
&&\lesssim  \|\Psi_1\|_{H^s(\lambda)}^2\|\Psi_2\|_{H^s(\lambda)}^2\|\Psi_3\|_{B^{s-\sigma}_{r,2}(\lambda)}^2 +\|\Psi_2\|_{H^s(\lambda)}^2\|\Psi_3\|_{H^s(\lambda)}^2\|\Psi_1\|_{B^{s-\sigma}_{r,2}(\lambda)}^2\nonumber\\
&&\quad +\|\Psi_3\|_{H^s(\lambda)}^2\|\Psi_1\|_{H^s(\lambda)}^2\|\Psi_2\|_{B^{s-\sigma}_{r,2}(\lambda)}^2
\label{cl3}
\eea
where we used the Minkowski and Cauchy-Schwarz inequalities. Next, from \fref{estimApsi}, \fref{cl3} and the H\"older inequality, one deduces that, if $\Psi\in \mathcal B_M$,
\bea
\|S(\Psi)\|_{X_T} &\leq& C_1\|\Phi\|_{H^s(\lambda)}+C_2\|\Psi\|_{L^\infty_TH^s(\lambda)}^2\|\Psi\|_{L^1_TB^{s-\sigma}_{r,2}(\lambda)}\nonumber\\
&\leq& C_1\|\Phi\|_{H^s(\lambda)}+C_2(2T)^{(q-1)/q}\|\Psi\|_{L^\infty_TH^s(\lambda)}^2\|\Psi\|_{L^q_TB^{s-\sigma}_{r,2}(\lambda)}\nonumber\\
&\leq&C_1\|\Phi\|_{H^s(\lambda)}+C_2(2T)^{(q-1)/q}M^3.\label{resu1}
\eea
Similarly, from \fref{cl3}, it can be deduced that,  for all $\Psi$, $\widetilde \Psi\in \mathcal B_M$,
$$\|F(\Psi)-F(\widetilde \Psi)\|_{L^1_TH^s(\lambda)}\leq C_3 (2T)^{(q-1)/q}\left(\|\Psi\|_{X_T}^2+\|\widetilde \Psi\|_{X_T}^2\right)\|\Psi-\widetilde\Psi\|_{X_T}$$
and then, from \fref{defA} and  Proposition \ref{Strichartz}, that
\be
\label{resu2}
\|S(\Psi)-S(\widetilde \Psi)\|_{X_T} \leq C_4 T^{(q-1)/q}M^2\|\Psi-\widetilde \Psi\|_{X_T}.
\ee
{}From these estimates one can conclude the proof. Indeed, choosing $T>0$ small enough such that
$$C_2(2T)^{(q-1)/q}M^3<\frac{M}{2}\quad \mbox{and}\quad C_4 T^{(q-1)/q}M^2<\frac{1}{2},$$
then from $M\geq 2C_1\|\Phi\|_{H^s(\lambda)}$, \fref{resu1} and \fref{resu2} one deduces that $S$ is a contraction mapping on $\mathcal B_M$, which proves Theorem \ref{main-theorem2}.

\bs
Let us prove the claim \fref{cl2}. In fact, for the sake of simplicity, we shall only prove this estimate for $u_1=\overline{u_2}=u_3=u$, the general case can be proved very similarly. We have then to prove that
\be
\label{cl4}
\left\|{\mathcal V}(|u|^2)u\right\|_{H^s}\lesssim \|u\|_{H^s}^2\|u\|_{B^{s-\sigma}_{r,2}}.
\ee
This inequality will result from paradifferential calculus. Let us first fix a few parameters. Let $\omega$ be such that $2/r-\omega=s-\sigma$, i.e. 
$$\omega=\frac{2}{r}+\sigma-s=\frac{1}{2r}+\frac{3}{4}-s.$$ {}From \fref{relation} and $3/8<s<1/2$, one can deduce that $r>2$ and $0<\omega<s$. Now, let us fix $p_1,p_2\in (2,+\infty)$ and $p_3,p_4\in(1,2)$ such that 
$$0<\frac{2}{p_2}<\min\left(\omega,\frac{2}{r}\right),\qquad \frac{1}{p_1}+\frac{1}{p_2}=\frac{1}{2},\qquad \frac{1}{p_3}=\frac{1}{p_1}+\frac{1}{2},\qquad \frac{1}{p_4}=\frac{1}{p_2}+\frac{1}{2}.$$
Paraproduct estimates and Besov embedding theorems \cite{Danchin,RS} give
\begin{align}
\|{\mathcal V}(|u|^2)u\|_{B^s_{2,2}}&\lesssim\|{\mathcal V}(|u|^2)\|_{L^\infty}\|u\|_{B^s_{2,2}}+\|{\mathcal V}(|u|^2)\|_{B^{s+\omega-2/p_2}_{p_1,2}}\|u\|_{B^{2/p_2-\omega}_{p_2,\infty}}\nonumber\\
&\lesssim\|{\mathcal V}(|u|^2)\|_{\dot B^{2/p_4}_{p_4,1}}\|u\|_{B^s_{2,2}}+\|{\mathcal V}(|u|^2)\|_{B^{s+\omega-2/p_2}_{p_1,2}}\|u\|_{B^{2/r-\omega}_{r,2}},
\label{para1}
\end{align}
where $\dot B^{s}_{p,q}$ denotes the homogeneous Besov space (\cite{BELO}). Let us now estimate ${\mathcal V}(u)$. The operator of convolution with $1/|x|$ in $\RR^2$ is of order -1 in {\em homogeneous} Besov spaces, thus we have
\begin{align}
\|{\mathcal V}(|u|^2)\|_{\dot B^{2/p_4}_{p_4,1}}&\lesssim \||u|^2\|_{\dot B^{2/p_4-1}_{p_4,1}}\nonumber\\
&\lesssim \||u|^2\|_{B^{2/p_4-1}_{p_4,1}}\nonumber\\
& \lesssim \|u\|_{B^{s}_{2,2}}\|u\|_{B^{2/p_4-1-s}_{p_2,2}}\nonumber\\
& \lesssim \|u\|_{B^{s}_{2,2}}\|u\|_{B^{2/r-s}_{r,2}}\lesssim \|u\|_{B^{s}_{2,2}}\|u\|_{B^{2/r-\omega}_{r,2}}
\label{para2}
\end{align}
(here we used that $0<2/p_4-1<s$ and that $0<\omega<s$), and
\begin{align}
\|{\mathcal V}(|u|^2)\|_{B^{s+\omega-2/p_2}_{p_1,2}}&\lesssim \|{\mathcal V}(|u|^2)\|_{L^{p_1}}+\|{\mathcal V}(|u|^2)\|_{\dot B^{s+\omega-2/p_2}_{p_1,2}} \nonumber\\
&\lesssim \||u|^2\|_{L^{p_3}}+\||u|^2\|_{\dot B^{s+\omega-2/p_2-1}_{p_1,2}}\nonumber\\
&\lesssim \|u\|_{L^2}\|u\|_{L^{p_1}}+ \|u\|_{\dot B^{s-2/p_2}_{p_1,2}}\|u\|_{\dot B^{\omega-1}_{\infty,\infty}}\nonumber\\
&\lesssim \|u\|_{B^{2/p_2}_{2,2}}^2+\|u\|_{\dot B^{s}_{2,2}}\|u\|_{\dot B^{\omega}_{2,\infty}}\nonumber\\
&\lesssim\|u\|_{B^s_{2,2}}^2.
\label{para3}
\end{align}
Here we used that $\omega-1<0$, that ${2\over p_{2}} < \omega <s$ and Hardy-Littlewood-Sobolev which yields $\|{\mathcal V}(u)\|_{L^{p_1}}\lesssim \|u\|_{L^{p_3}}$. Finally \fref{para1}, \fref{para2} and \fref{para3} give (since $B^s_{2,2}$ is identified with $H^s$)
$$\|{\mathcal V}(|u|^2)u\|_{H^s}\lesssim \|u\|_{H^s}^2\|u\|_{B^{2/r-\omega}_{r,2}}\,.$$
Since $2/r-\omega=s-\sigma$, this proves \fref{cl4} and the proof of Theorem \ref{main-theorem2} is complete.

\end{proof}

 \section{The case $\Omega$ bounded: proof of Theorem \ref{main-theorem1}}\label{boundeddomain}   

In this section, we prove Theorem \ref{main-theorem1}. Let us first provide some estimates on the products of functions in $H^s_A$ spaces (recall Definition \eqref{Hsnorm}), which will be used in the nonlinear analysis of \fref{dirac2}, \fref{poisson2}. The following technical lemma is proved in Appendix \ref{appB}.
\begin{lemma}
\label{lemmaproducts}
\begin{enumerate}[(i)]
\item Let $u\in H^{s}_A$ and $v\in L^\infty\cap H^{\sigma}_A$ with $0<s<5/2$ and $\sigma \geq \max(s,1)$. In the case $s=1$, we assume additionnally that $\sigma>1$. Then $uv \in H^s_A$ and
\be
\label{prod1}
\|uv\|_{H^s_A}\leq C(s,\sigma)\,\|u\|_{H^s_A}(\|v\|_{H^{\sigma}_A}+\|v\|_{L^\infty}).
\ee
\item Let $u\in H^s_A$ and $v\in H^s_A$ with $\frac{1}{2}\leq s< 1$. Then $uv \in H^{2s-1}_A$ and
\be
\label{prod2}
\|uv\|_{H^{2s-1}_A}\leq C(s)\,\|u\|_{H^s_A}\|v\|_{H^s_A}.
\ee
\end{enumerate}
\end{lemma}
\bs

\begin{remark}
The limitation $s<\frac{5}{2}$ is sharp in this lemma since, for $s\geq \frac{5}{2}$, the product of two functions in $H^s_A$ does not necessarily belong to this space. Indeed, according to \cite{grisvard}, we have $$H^s_A=\left\{u\in H^s(\Omega)\cap H^1_0(\Omega):\,\Delta u\in H^{s-2}_0(\Omega)\right\}\quad \mbox{for}\quad \frac{5}{2}\leq s<\frac{7}{2}.$$
Now, for $u,v\in H^s_A$, the function $\Delta (uv)$ may not belong to $H^{s-2}_0(\Omega)$. For instance, consider $\Omega=]0,1[$ and $u(x)=\sin(\pi x)$. We have $u\in H^s_A$ but $u^2\not\in H^{5/2}_A$. Indeed, $(u^2)''(x)=2\pi^2\cos(2\pi x)\not\in H^{1/2}_0(0,1)$ since $\rho^{-1/2}(u^2)''\not \in L^2(0,1)$ with $\rho(x)=\min{(x,1-x)}$ (see the characterization of $H^{1/2}_0(0,1)$ in Appendix \ref{appB}).
\end{remark}

\bs
We are now ready to prove our main result concerning the problem in the case of a bounded domain $\Omega$.
\begin{proof}[Proof of Theorem \pref{main-theorem1}]
The operator $\widetilde H_0$ is self-adjoint on its domain $D(\widetilde H_0)=H^{1}_A$. Let $\widetilde K(t)=e^{-it\widetilde H_0}$ be the unitary group generated by this operator, which can also be written as
$$
\widetilde K(t)=\cos(t(-\Delta)^{1/2}){\mathrm I}_2-i\sin(t(-\Delta)^{1/2})\sigma_1\,,
$$
the matrix $\sigma_1$ being defined by \eqref{paulimatrices}.
Since, for all $s\geq 0$, $H^s_A$ is defined as the domain of $A^{s/2}$, $\widetilde K(t)$ is unitary on this space $H^s_A$. The problem \fref{dirac2}, \fref{poisson2}  can be rewritten under the following mild formulation
\beq \label{integral-form2}
\Psi_j(t)=\widetilde K(t) \Phi_j + \int_0^t \widetilde K(t-t')F_j(\Psi(t'))dt'\eeq
with $F_j(\Psi)=V(\Psi)\Psi_{j}$. Let $\frac{1}{2}\leq s<\frac{5}{2}$ and $\Phi=(\Phi_j)_{j\in\NN}\in H^s_A(\lambda)$. In order to solve \fref{integral-form2}, for $T>0$ we introduce the mapping $\widetilde S$ on $C^0([-T,T],H^{s}_A(\lambda))$ defined by 
 \begin{equation}\label{Adef2}
 \widetilde S(\Psi)_j(t)=\widetilde K(t) \Phi_j + \int_0^t\widetilde K(t-t')F_j(\Psi(t'))dt'.
 \end{equation}
 For $M > 2\|\Phi\|_{H^s_A(\lambda)}$, let
 \beq\label{B_Mdef}
 \mathcal B_M= \{ \Psi \in  C^0([-T,T],H^{s}_A(\lambda))\,:\; \max_{t\in[-T,T]}\|\Psi(t)\|_{H^{s}_A(\lambda)}\leq M\}.\eeq
To prove Theorem \ref{main-theorem1}, it is again sufficient to show that $\widetilde S$ is a contraction mapping on $\mathcal B_M$ for $T$ small enough.  Let $\Psi(t)=(\Psi_j(t))_{j\in \NN} \in \mathcal B_M$, then $\widetilde K$ being unitary on $H^s_A$ implies that
$$\max_{t\in[-T,T]}\|\widetilde S(\Psi)(t)\|_{H^s_A(\lambda)} \leq \|\Phi\|_{H^s_A(\lambda)}+\int_0^T\|F(\Psi)(\tau)\|_{H^s_A(\lambda)}d\tau.$$
Now, we claim that, for all $\Psi\in H^s_A(\lambda)$, one has
\be
\label{claim2}\|F(\Psi)\|_{H^s_A(\lambda)}\lesssim \|\Psi\|_{H^s_A(\lambda)}^3.
\ee
Assuming this claim, we deduce that, for $\Psi \in \mathcal B_M$, there exists a constant $C_1 > 0$ such that
 $$ \max_{t\in[-T,T]}\|\widetilde S(\Psi)(t)\|_{H^s_A(\lambda)} \leq \|\Phi\|_{H^s_A(\lambda)}+2C_1T\max_{t\in[-T,T]}\|\Psi(t)\|_{H^s_A(\lambda)}^3.$$
Similarly, for $\Psi, \widetilde\Psi \in \mathcal B_M$, one can prove that
$$  \max_{t\in[-T,T]}\|(\widetilde S(\Psi)-\widetilde S(\widetilde \Psi))(t)\|_{H^s_A(\lambda)} \leq 2C_2M^2T\max_{t\in[-T,T]}\|(\Psi-\widetilde \Psi)(t)\|_{H^s_A(\lambda)}.$$
 Then, if we choose $T$ such that 
 $$\max(C_1,C_2)M^2T\leq \frac14,$$
 one deduces that $\widetilde S$ is a contraction mapping on $\mathcal B_M$ and Theorem \ref{main-theorem1} is proved. 

 It remains to prove \fref{claim2}. Let 
 $$\sigma=\left\{\begin{array}{ll}2s \,&\mbox{if }\, 1/2\leq s< 1,\\
 3/2\,&\mbox{if }\,  s=1,\\s+1\,&\mbox{if }\, s>1.\end{array}\right.$$
In all cases, the pair $(s,\sigma)$ fulfills the conditions of Lemma \ref{lemmaproducts} {\em (i)}. Hence, using Definition \ref{SpacesDef} and Lemma \ref{lemmaproducts}, we get
 \bea
 \|F(\Psi)\|_{H^s_A(\lambda)} &=& \left(\sum_j\lambda_j\|V(\Psi)\Psi_j\|_{H^s_A}^2\right)^{1/2}\nonumber\\
 &\lesssim& \left(\sum_j\lambda_j(\|V(\Psi)\|_{H^\sigma_A}+\|V(\Psi)\|_{L^\infty})^2\|\Psi_j\|_{H^s_A}^2\right)^{1/2}\nonumber\\
 &&=(\|V(\Psi)\|_{H^\sigma_A}+\|V(\Psi)\|_{L^\infty}) \|\Psi\|_{H^s_A(\lambda)}.\label{est1}
 \eea
 Next, from the regularization properties of \fref{poisson2},  we deduce
 $$  \|V(\Psi)\|_{H^\sigma_A}\leq \sum_j\lambda_j \||\Psi_j|^2\|_{H^{\sigma-1}_A}.$$
In the case $\frac{1}{2}\leq s<1$, \fref{prod2} gives
$$\||\Psi_j|^2\|_{H^{\sigma-1}_A}=\||\Psi_j|^2\|_{H^{2s-1}_A}\lesssim \|\Psi_j\|_{H^s_A}^2.$$
In the case $s=1$, \fref{prod2} gives
$$\||\Psi_j|^2\|_{H^{\sigma-1}_A}=\||\Psi_j|^2\|_{H^{1/2}_A}\lesssim \|\Psi_j\|_{H^{3/4}_A}^2\leq \|\Psi_j\|_{H^1_A}^2.$$
Finally, in the case $s>1$, $H^s\hookrightarrow L^\infty$. Then, \fref{prod1} gives
$$\||\Psi_j|^2\|_{H^{\sigma-1}_A}=\||\Psi_j|^2\|_{H^{s}_A}\lesssim \|\Psi_j\|_{H^{s}_A}^{2} .$$
In all cases, we have then
\begin{equation}
\label{new1} \|V(\Psi)\|_{H^\sigma_A}\lesssim \sum_j\lambda_j \|\Psi_j\|_{H^s_A}^2=\|\Psi\|^2_{H^s_A(\lambda)}.
\end{equation}

It remains to estimate the $L^\infty$ norm of $V(\Psi)$. We will in fact prove the following estimate:
\begin{equation}
\label{new2}
\|V(\Psi)\|_{L^\infty}\leq \|\Psi\|^2_{H^{1/2}_A(\lambda)}
\end{equation}
which, together with \fref{est1} and \fref{new1}, yields \fref{claim2}. To prove \eqref{new2}, we will need a lemma on the kernel of $A^{1/2}$.
\begin{lemma}
\label{lemmakernel}
Let $f \geq 0$ and $f\in L^2(\Omega)$. Let $u$ be the solution of $A^{1/2}u=f$, where $A=-\Delta_\Omega$ is the opposite of the Dirichlet Laplacian on $\Omega$. Then $u\in H^1_0(\Omega)$ and, for almost all $x\in \Omega$, we have
\begin{equation}\label{estikernel}
0\leq u(x)\leq \frac{1}{2\pi}\int_{\Omega}\frac{f(y)}{|x-y|}dy.
\end{equation}
\end{lemma}
\begin{proof}[Proof of Lemma \pref{lemmakernel}]
The fact that $u$ belongs to $H^1_0(\Omega)=H^1_A$ is obvious, from the spectral characterization of $u$, since
$$\|u\|_{H^1_0}^2=\sum_{p\in \NN}\mu_p \,|\langle u,e_p\rangle|^2=\sum_{p\in \NN}|\langle A^{1/2}u,e_p\rangle|^2=\sum_{p\in \NN}|\langle f,e_p\rangle|^2=\|f\|_{L^2}^2<+\infty.$$
From \cite{ct}, we know that the problem $A^{1/2}u=f$ has the following harmonic extension in the half-cylinder $\mathcal C:=\Omega\times (0,+\infty)$. Consider the function $v(x,z)$ satisfying the mixed boundary-value problem
\begin{equation}
\left\{
\begin{array}{rcrl}
-\Delta v&=&0\quad &\mbox{in } \mathcal C,\\[2mm]
v&=&0\quad &\mbox{on } \pa\Omega\times (0,+\infty),\\[2mm]
\ds \frac{\pa v}{\pa \nu}&=&f\quad &\mbox{on } \Omega\times \{0\},
\end{array}
\right.
\end{equation}
where $\nu$ is the unit outer normal to $\mathcal C$ at $\Omega\times \{0\}$, then we have $u=v(\cdot,0)$. The property \eqref{estikernel} is then a direct consequence of the maximum principle. Indeed, consider the function $\overline u(x,z)$ defined by
$$\overline u(x,z)=\frac{1}{2\pi}\int_{\Omega}\frac{f(y)}{\sqrt{|x-y|^2+z^2}}dy=\frac{1}{2\pi}\int_{\RR^2}\frac{\widetilde f(y)}{\sqrt{|x-y|^2+z^2}}dy,$$
where $\widetilde f$ denote the extension of $f$ by zero outside $\Omega$. It can be checked that this function satisfies $-\Delta \overline u=2\widetilde f(x)\,\delta_{z=0}$ in the sens of distributions $\mathcal D(\RR^3)$, which yields
\begin{equation}
\left\{
\begin{array}{rcrl}
-\Delta \overline u&=&0\quad &\mbox{in } \mathcal C,\\[2mm]
\overline u&\geq&0\quad &\mbox{on } \pa\Omega\times (0,+\infty),\\[2mm]
\ds \frac{\pa \overline u}{\pa \nu}&=&f\quad &\mbox{on } \Omega\times \{0\}.
\end{array}
\right.
\end{equation}
Therefore, the maximum principle implies $0\leq v(x,z)\leq \overline u(x,z)$. Setting $z=0$ in this inequality gives \eqref{estikernel}.
\end{proof}
Using this lemma, we proceed as in the whole space case (see \eqref{MH2}) to conclude. Indeed, by the Hardy inequality in $\dot H^{1/2}(\RR^2)$ (see e.g. \cite{tao}), we get
 \begin{align*}
 \|V(\Psi)\|_{L^\infty}&=\frac{1}{2}\left\|\sum_{j\in \NN}\lambda_j A^{-1/2}(|\Psi_j|^2)\right\|_{L^\infty}\leq\frac{1}{2}\sum_{j\in \NN}\lambda_j \left\|A^{-1/2}(|\Psi_j|^2)\right\|_{L^\infty}\\
&\lesssim \sum_{j\in \NN}\lambda_j \supess_{x\in \RR^2}\int_{\RR^2}\frac{|\widetilde \Psi_j(y)|^2}{|x-y|}dy
=\sum_{j\in \NN}\lambda_j \supess_{x\in \RR^2}\int_{\RR^2}\frac{|\widetilde \Psi_j(x-y)|^2}{|x|}dy\\
&\lesssim\sum_{j\in \NN}\lambda_j\supess_{x\in \RR^2}\int_{\RR^2}|(-\Delta_y)^{1/4}\widetilde \Psi_j(x-y)|^2dy\\
&\quad = \sum_{j\in \NN}\lambda_j\|\widetilde \Psi_{j}\|_{\dot H^{1/2}(\RR^2)}^2\\
&\quad \leq \sum_{j\in \NN}\lambda_j\|\widetilde \Psi_{j}\|_{H^{1/2}(\RR^2)}^2\lesssim \sum_{j\in \NN}\lambda_j\|\Psi_{j}\|_{H^{1/2}_A}^2 = \|\Psi\|_{H^{1/2}_A(\lambda)}^2,
 \end{align*}
 where we have denoted by $\widetilde \Psi_j$ the extension of $\Psi_j$ by zero outside $\Omega$. We have proved \eqref{new2} and the proof of Theorem \ref{main-theorem1} is complete.
 \end{proof}

\appendix
\section{The Poisson equation with surface densities}\label{appendix}
In this appendix, we sketch a derivation of the confined Poisson equation, \eqref{poisson} or \fref{poisson2}, associated to a confined electron gas in a plane domain $\Omega$. Such derivation was done rigorously by asymptotic analysis for the nonlinear Schr\"odinger-Poisson system in \cite{BAMP,delebecque}. Here, we only show how \eqref{poisson} or \fref{poisson2} can be formally obtained as the trace of the 3D Poisson equation in the case of a surface density. 

Hence, the starting point of this derivation is the 3D Poisson equation satisfied by the electric potential $V^{3D}$. We consider a charge density under the form $n^{3D}(x,z)=n(x)\delta(z)$ for $x\in \Omega\subseteq \RR^2$, $z\in \RR$, where $\delta(z)$ denotes the Dirac mass. As in the rest of the paper, we distinguish two cases: when $\Omega=\RR^2$ and when $\Omega$ is a bounded regular domain of $\RR^2$.

\subsection{Case $\Omega=\RR^2$}
When the electron gas is allowed to move in all the plane $\RR^2\times \{0\}\subset \RR^3$, the Poisson equation takes the integral form
$$
V^{3D}(x,z)=\frac{1}{4\pi\sqrt{|x|^2+z^2}}\ast (n(x)\delta(z))=\int_{\RR^2}\frac{n(x')}{4\pi\sqrt{|x-x'|^2+z^2}}dx'.
$$
Therefore, the effective potential seen by the particles is
\begin{equation*}
V(x)=V^{3D}(x,0)=\int_{\RR^2}\frac{n(x')}{4\pi\sqrt{|x-x'|^2}}dx'=\frac12(-\Delta)^{-1/2}n.
\end{equation*}

\subsection{Case $\Omega$ bounded} We assume now that the electron gas is only allowed to occupy a device modeled by a bounded domain $\Omega\times \{0\}\subset \RR^3$. Simultaneously, one has to prescribe boundary conditions for $V^{3D}$. We shall consider only the simplest case where the boundary is assumed to be connected to a cylindrical perfect conductor. Then the 3D potential satisfies
\begin{equation}\label{poisson3D}
(-\pa_z^2-\Delta) V^{3D}= n(x)\delta(z) \; \text{ on } \Omega\times \RR,\quad \mbox{ where }\Delta=\pa_{x_1}^2+\pa_{x_2}^2,
\end{equation}
with the boundary conditions
$$
V^{3D}(x,z)=0, \;\text{ on } \partial\Omega\times \RR\quad \mbox{and}\quad V^{3D}\to 0\mbox{ as }z\to \pm \infty
$$
Let $(e_p)_{p\in\NN}$ denote an orthonormal basis of eigenvectors of $-\Delta$ on $\Omega$ with Dirichlet boundary conditions, and let $(\mu_p)_{p\in\NN}$  be the associated eigenvalues. For any $p\geq 0$, the orthogonal projection of \eqref{poisson3D} on 
$e_p$ gives
\begin{equation*}
-\partial_{z}^2V_p^{3D}+\mu_pV_p^{3D}=n_p\delta(z),
\end{equation*}
where $\displaystyle V_p^{3D}(z)=\int_{\Omega}V^{3D}(x,z)e_p(x)dx$ and $\displaystyle n_p=\int_{\Omega}n(x)e_p(x)dx$. One can solve explicitly the last equation: a straightforward  computation gives
\begin{equation*}
V^{3D}(x,z)=\frac12\sum_{p=0}^{+\infty}n_pe_p(x)\frac{e^{-\sqrt{\mu_p}|z|}}{\sqrt{\mu_p}}, \quad z\in \RR, x\in \Omega.
\end{equation*}
Therefore,
\begin{equation*}
V(x)=V^{3D}(x,0)=\frac12\sum_{p=0}^{+\infty}(\mu_p)^{-1/2}n_pe_p(x)=\frac12(-\Delta)^{-1/2}n.
\end{equation*}
\section{Fractional order Sobolev spaces on $\Omega$}
\label{appB}
In this Appendix, we prove the technical Lemma \ref{lemmaproducts} used in the proof of Theorem \ref{main-theorem1}. Before that, let us first recall some useful facts on the domain $H^s_A$ of the fractional Laplacian with Dirichlet boundary conditions and on fractional order Sobolev spaces. For $m\in \NN$, $H^m(\Omega)$ denotes the usual Sobolev space and 
$$H^m_0(\Omega)=\{u\in H^m(\Omega)\,:\; \pa^\alpha u_{\mid\pa\Omega}=0,\quad0\leq |\alpha|<m\}\,$$
where we have used the standard notations for the multiindex $\alpha=(\alpha_1,\alpha_2)$, $|\alpha|=\alpha_1+\alpha_2$ and $\pa^\alpha u=\pa^{\alpha_1}_{x_1}\pa^{\alpha_2}_{x_2}u$. For $s\in \RR_+\setminus \NN$, the fractional Sobolev space $H^s(\Omega)$ is the space defined by interpolation \cite{LIONS}
$$
H^s(\Omega)=[H^m(\Omega),L^2(\Omega)]_\theta, \quad (1-\theta)m=s, \quad m\in \NN, \quad 0<\theta<1,
$$
equipped with the norm
\begin{equation}
\label{fd1}
\|u\|_{H^s(\Omega)}^2=\int\int_{\Omega\times\Omega}\frac{|u(x)-u(y)|^2}{|x-y|^{2+2s}}dxdy\qquad \mbox{for }0<s<1,
\end{equation}
and
\begin{equation}
\label{fd2}
\|u\|_{H^s(\Omega)}^2=\|u\|_{H^{m}}^2+\sum_{|\alpha|=m}\|\pa^\alpha u\|_{H^{s-m}}^2\qquad \mbox{for }s>1,\quad m=\lfloor s\rfloor.
\end{equation}
Moreover, the fractional Sobolev space $H^s_0(\Omega)$ is also defined by interpolation by\footnote{In the case $s\in \frac{1}{2}+\NN$, this space is called the Lions-Magenes space and is often denoted $H^s_{00}(\Omega)$ instead of $H^s_0(\Omega)$ as here.}
$$H^s_0(\Omega)=[H^m_0(\Omega),L^2(\Omega)]_\theta, \quad (1-\theta)m=s, \quad m\in \NN, \quad 0<\theta<1,$$
and $\|\cdot\|_{H^s_0}$ denotes the associated interpolation norm.

\bs
Our main result in the case of bounded domains is formulated using the domain $H^s_A$ of the operator $A^{s/2}$ for $s<\frac{5}{2}$ (see its definition \fref{Hsnorm}). According to  \cite{LIONS,grisvard}, it can also be characterized as follows:
\begin{enumerate}[(a)]
\item For $0\leq s<\frac{3}{2}$, we have $H^s_A=H^s_0(\Omega)$.\\[-3mm]
\item For $\frac{3}{2}\leq s<\frac{5}{2}$,  we have $H^s_A=H^s(\Omega)\cap H^1_0(\Omega)$.
\end{enumerate}
Moreover, for all $0\leq s<\frac{5}{2}$ and $s\neq \frac12$, the $H^s_A$ norm is equivalent to the $H^s(\Omega)$ norm. In the specific case $s=\frac12$, the $H^{1/2}_A$ norm is equivalent to the $H^s_0(\Omega)$ norm, which is {\em not} equivalent to the $H^{1/2}(\Omega)$ norm. 

Indeed, for $0\leq s<\frac{3}{2}$, we recall the following equivalent characterization of the $H^s_0(\Omega)$ space, see for instance \cite{LIONS,trelat,tartar}. It will be useful in particular in the proof of the technical Lemma \ref{lemmaproducts}.
\begin{enumerate}[(a)']
\item For all $0\leq s<\frac{3}{2}$, $H^s_0(\Omega)$ is characterized as the space of functions $u\in H^s(\Omega)$ such that their extension by 0 outside $\Omega$, denoted by $\widetilde u$, belong to $H^s(\RR^2)$, and the norm $\|\cdot\|_{H^s_{0}(\Omega)}$ is equivalent to the norm $u\mapsto \|\widetilde u\|_{H^s(\RR^2)}$.
\item For $0\leq s<\frac{1}{2}$, we have $H^s_0(\Omega)=H^s(\Omega)$, and the associated norms are equivalent.
\item For $\frac{1}{2}<s<\frac{3}{2}$, we have 
$$H^s_0(\Omega)=\{u\in H^s(\Omega)\,:\; u_{\mid\pa\Omega}=0\}\,$$
and the norm $\|\cdot \|_{H^s_0(\Omega)}$ is equivalent to the norm $\|\cdot \|_{H^s(\Omega)}$ defined above.
\item For $s=\frac{1}{2}$, the space $H^{1/2}_0(\Omega)$ is characterized as
$$H^{1/2}_0(\Omega)=\left\{u\in H^{1/2}(\Omega)\,:\; \rho^{-1/2}\,u\in L^2(\Omega)\right\},$$where $\rho(x)$ denotes the distance function to $\pa\Omega$. Moreover, the norm $\|\cdot \|_{H^s_0(\Omega)}$ is equivalent to the norm
$$\left(\|u\|_{H^s(\Omega)}^2+\|\rho^{-1/2}u\|_{L^2(\Omega)}^2\right)^{1/2}.$$
\end{enumerate}

\bs
We are now ready to prove the technical Lemma \pref{lemmaproducts}.
\begin{proof}[Proof of Lemma \pref{lemmaproducts}] Let us first prove {\em (i)} when $0<s< 3/2$. With no loss of generality, we can assume that $\sigma<3/2$. Hence, from the above Item (a), the spaces $H^s_A$ and $H^\sigma_A$ can be respectively identified with $H^s_0$ and $H^\sigma_0$. Denote by $\widetilde u$ and $\widetilde v$ the extensions of $u$ and $v$ by zero outside $\Omega$. By Item (a)', we have $\widetilde u\in H^s(\RR^2)$ and $\widetilde v\in H^s(\RR^2)$. Therefore, we have
$$\|uv\|_{H^s_A} \lesssim \|\widetilde {uv}\|_{H^s(\RR^2)}=\|\widetilde u\,\widetilde v\|_{H^s(\RR^2)}.$$
If $1<s<\frac{3}{2}$ then $H^s(\RR^2)$ is an algebra and the usual estimate for the product of $H^s(\RR^2)$ functions \cite{ALG} yields
\beq
\label{algebra1}
\|uv\|_{H^s_A}\lesssim \|\widetilde u\|_{H^s(\RR^2)}\|\widetilde v\|_{H^s(\RR^2)}\lesssim \|u\|_{H^s_A}\|v\|_{H^s_A},
\eeq
which yields \fref{prod1} since $\sigma\geq s$. If $s\leq 1$, from the Littlewood-Paley theory and the Mikhlin-H\"ormander multiplier theorem \cite{ALG,SCH} one deduces
\be
\label{inter}
\|\widetilde u\,\widetilde v\|_{H^s(\RR^2)}\lesssim \|\widetilde u\|_{L^{p_1}(\RR^2)}\|\widetilde v\|_{W^{s,p_2}(\RR^2)}+\|\widetilde u\|_{H^s(\RR^2)}\|\widetilde v\|_{L^\infty(\RR^2)}
\ee
for all $\frac{1}{p_1}+\frac{1}{p_2}=\frac{1}{2}$ and $2<p_1,\,p_2<+\infty$. Here $W^{s,p}(\RR^2)$ denotes the usual Sobolev spaces on $\RR^2$. If $s<1$, take $p_1=\frac{2}{1-s}$ and $p_2=\frac{2}{s}$ so that by Sobolev embeddings
\bee
\|uv\|_{H^s_A} &\lesssim& \|\widetilde u\|_{H^s(\RR^2)}\|\widetilde v\|_{H^1(\RR^2)}+\|\widetilde u\|_{H^s(\RR^2)}\|\widetilde v\|_{L^\infty(\RR^2)}\\
&\lesssim& \|u\|_{H^s_A}\|v\|_{H^1_A}+\|u\|_{H^s_A}\|v\|_{L^\infty}\leq\|u\|_{H^s_A}(\|v\|_{H^\sigma_A}+\|v\|_{L^\infty}),
\eee
since $\sigma\geq 1$. If $s=1$, then our assumption is that $1<\sigma<3/2$. Let $p_2=\frac{2}{2-\sigma}$ and $p_1=\frac{2p_2}{p_2-2}$, implying that $H^\sigma(\RR^2)\hookrightarrow W^{s,p_2}(\RR^2)$ and $H^s(\RR^2)\hookrightarrow L^{p_1}(\RR^2)$. Hence, \fref{inter} yields
$$
\|uv\|_{H^1_A} \lesssim \|\widetilde u\|_{H^1(\RR^2)}\|\widetilde v\|_{H^\sigma(\RR^2)}+\|\widetilde u\|_{H^1(\RR^2)}\|\widetilde v\|_{L^\infty(\RR^2)}
\lesssim \|u\|_{H^1_A}(\|v\|_{H^\sigma_A}+\|v\|_{L^\infty}).
$$

\bs
Let us prove {\em (i)} when $3/2\leq s< 5/2$. Again, with no loss of generality, we can assume that $\sigma<5/2$. Here, from the above Item (b), we have $H^s_A=H^s(\Omega)\cap H^1_0(\Omega)$ and $H^\sigma_A=H^\sigma(\Omega)\cap H^1_0(\Omega)$. Since $u$ and $v$ vanish on $\pa \Omega$, it is clear that $uv$ also vanishes on $\pa \Omega$. Moreover, it is well-known \cite{RS} that $H^s(\Omega)$ is an algebra for $s> 1$ and that 
\beq
\label{algebra2}
\|uv\|_{H^s(\Omega)}\lesssim \|u\|_{H^s(\Omega)}\|v\|_{H^s(\Omega)}.
\eeq
This is enough to deduce \eqref{prod1}.  

\bs
Let us now prove {\em (ii)}. If $s=\frac{1}{2}$, the estimate stems from the Sobolev embedding $H^{1/2}(\Omega)\hookrightarrow L^4(\Omega)$ and from H\"older:
$$\|uv\|_{L^2}\leq \|u\|_{L^4}\|v\|_{L^4}\lesssim \|u\|_{H^{1/2}}\|v\|_{H^{1/2}}.$$
Consider now the case $\frac{1}{2}<s<1$. One has $0<2s-1<1$, thus, again by Item (a) the spaces $H^{2s-1}_A$ and $H^s_{A}$ can be respectively identified with $H^{2s-1}_0(\Omega)$ and $H^{s}_0(\Omega)$ and by Mikhlin-H\"ormander,
\bee
\|uv\|_{H^{2s-1}_0(\Omega)} &\lesssim& \|\widetilde u\,\widetilde v\|_{H^{2s-1}(\RR^2)}\\
&\lesssim& \|\widetilde u\|_{W^{{2s-1},p_1}(\RR^2)}\|\widetilde v\|_{L^{p_2}(\RR^2)}+\|\widetilde u\|_{L^{p_2}(\RR^2)}\|\widetilde v\|_{W^{{2s-1},p_1}(\RR^2)}\\
&\lesssim& \|\widetilde u\|_{H^s(\RR^2)}\|\widetilde v\|_{H^s(\RR^2)}\lesssim \|u\|_{H^s_0(\Omega)}\|v\|_{H^s_0(\Omega)}.
\eee
Here we have chosen $p_1=\frac{2}{s}$ and $p_2=\frac{2}{1-s}$, so that, by Sobolev embeddings, $H^s(\RR^2)\hookrightarrow W^{{2s-1},p_1}(\RR^2)\cap L^{p_2}(\RR^2)$. The proof of the lemma is complete.
\end{proof}

\section*{Acknowledgment}
We acknowledge support by the ANR-FWF Project Lodiquas (ANR-11-IS01-0003). We would like to thank the anonymous referees for suggesting improvements of our proofs.


\end{document}